\documentclass[a4paper,11pt]{article}
\usepackage{amssymb}
\usepackage{amsmath}
\usepackage{amsthm}
\usepackage{graphicx}
\newtheorem{thm}{Theorem}[section]
\newtheorem{lem}[thm]{Lemma}

\newtheorem{defn}[thm]{Definition}

\newtheorem{exa}[thm]{Example}
\numberwithin{equation}{section}
\newcommand{\LLL}{L^2(\Omega)}

\newcommand{\va}{\varphi}
\newcommand{\ppp}{\partial}

\newcommand{\R}{\mathbb{R}}
\newcommand{\C}{\mathbb{C}} 
\newcommand{\N}{\mathbb{N}} 
\newcommand{\www}{\widetilde}

\allowdisplaybreaks
\title{Operator theoretic approach to the Caputo derivative and the fractional diffusion equations}
\author{Rudolf Gorenflo\thanks{Department of Mathematics and Informatics, 
Free University of Berlin, 
Arnimallee 3, 14195  Berlin, Germany,
e-mail: {\tt gorenflo@mi.fu-berlin.de}}, 
Yuri Luchko \thanks{Department of Mathematics, Physics, and Chemistry, Beuth University of Applied Sciences,   Luxemburger Str. 10, 
13353 Berlin, Germany,
e-mail: {\tt luchko@beuth-hochschule.de}}
and Masahiro Yamamoto
\thanks{ Department of Mathematical Sciences, The University
of Tokyo, Komaba, Meguro, Tokyo 153, Japan,
e-mail:
{\tt myama@ms.u-tokyo.ac.jp}}
}

\date{}
\begin{document}
\maketitle
\begin{abstract}
The Caputo time-derivative is usually defined pointwise 
for well-behaved functions, say, for continuously differentiable functions. 
Accordingly, 
in the theory of the partial fractional differential equations with the Caputo derivatives, the functional spaces where the solutions are looked for are often the spaces of the smooth functions that are too narrow.  In this paper, 
we introduce a suitable definition of the 
 Caputo derivative in the fractional Sobolev spaces and investigate it from the operator theoretic viewpoint.  In particular, some important equivalences of the norms related to the fractional integration and differentiation operators in the fractional Sobolev spaces are given.  
These results are then applied for proving  the maximal regularity of the 
solutions to some initial-boundary value problems for
the time-fractional diffusion equation with the Caputo derivative in the fractional Sobolev spaces.
\end{abstract}

\vspace{0.1cm}



%\baselineskip 18pt

%\noindent
%\bf \S1. 
\section{Introduction}

In a very recent survey paper \cite{Met14} that is devoted to the systems, which exhibit anomalous diffusion, about three hundreds references to the relevant works are given. Many of the cited publications deal with modeling of the anomalous diffusion with the continuous time random walks on the micro-level and with the fractional diffusion equations on the macro-level. The permanently growing number of publications devoted to anomalous diffusion and its modeling with the Fractional Calculus (FC) operators poses some challenges on the mathematical theory of FC and in particular on the theory of the partial differential equations of fractional order. This paper is devoted to one of these challenges, namely, to suggest  a suitable definition of the Caputo fractional derivative in the fractional Sobolev spaces, to consider its properties in these spaces, and to apply them for analysis of the fractional diffusion equations in the fractional Sobolev spaces.

For the theory of the FC operators, we refer the reader to 
the encyclopedia \cite{SKM}. The basic theory of the ordinary and partial fractional differential equations can be found e.g. in the monographs  
\cite{Diet},  \cite{Kil}, and \cite{P}.  We mention here also the papers \cite{Beck}, \cite{LuY}, 
 \cite{Lu1}-\cite{Lu3},  \cite{SY}, where some recent developments regarding the partial fractional differential equations are presented.  

In this paper, we deal with the fractional diffusion equation  
$$
\ppp_t^{\alpha}u(x,t) = -Lu(x,t) + F(x,t),  \ x\in \Omega \subset \R^n,\ 0<t\le T,   \eqno{(1.1)}
$$
where $-L$ is a differential operator of the elliptic type and 
 $\ppp_t^{\alpha}$ denotes the Caputo derivative that is usually 
defined by  the formula
$$
\ppp_t^{\alpha}u(x,t) = \frac{1}{\Gamma(1-\alpha)}\int^t_0
(t-s)^{-\alpha}\frac{\ppp u}{\ppp s}(x,s) ds, \quad 0 < t \le T, 
\quad 0 < \alpha < 1.            \eqno{(1.2)}
$$
To avoid switching between notations we consistently write
the operator symbol for fractional derivation with round $\partial$,
regardless of the number of independent variables.

In the formula (1.2), the Caputo derivative $\ppp_t^{\alpha}u$ is a  
derivative of the order $\alpha, \ 0 < \alpha < 1$. Still, in its definition the first 
derivative $\frac{\ppp u}{\ppp s}$ is involved that requires extra
regularity of the function $u$ and is meaningful only if $\frac{\ppp u}{\ppp s}$ exists in 
a suitable sense.

On the other hand, in many applications one has to deal with the non-differentiable functions and it is important to introduce a  weak solution  to the  fractional diffusion equation 
(1.1) in the case where  
$\frac{\ppp u}{\ppp t}$ does not exist in the usual sense (see e.g. \cite{LM} for the theory of the weak solutions of partial differential equations). 
For partial differential equations, the weak solutions are often constructed  in the Sobolev spaces (\cite{LM}). In this paper, we try to extend the theory of the weak solutions to partial differential equations in the Sobolev spaces 
to the fractional diffusion equation (1.1). 
The first problem which we have to overcome is to 
interpret the fractional Caputo derivative 
$\ppp_t^{\alpha}$ in the fractional Sobolev spaces and not by the pointwise definition
(1.2).  To the best knowledge of the authors, a solution to this problem  was 
not yet suggested in the literature.

It is worth mentioning that there are some publications (see e.g. \cite{Er}, \cite{Bangti} and the references there) devoted to the Riemann-Liouville fractional derivative  in the fractional Sobolev spaces. However, their approach via the Fourier transform is essentially different from the approach which 
we suggest in this paper for defining the Caputo derivative in the fractional Sobolev spaces. 

One of the main applications of the  fractional derivatives in the fractional Sobolev spaces is for introducing the weak or the generalized solution to the fractional differential equations. Of course, like in the theory of partial differential equations, different approaches can be used to attack this problem. In particular, in \cite{Lu2010}, a generalized solution to the initial-boundary-value problems for the  fractional diffusion equation in the sense of Vladimirov was introduced and analyzed. This generalized solution is a continuous function that is not necessarily differentiable. To construct the generalized solution, a formal solution in terms of the Fourier series with respect to the eigenfunctions  of the operator $L$ from (1.1) was employed. 

The same idea of the formal solution was used in \cite{SY} for constructing a 
weak solution to some  initial-boundary-value problems for an equation of the type (1.1) and for  proving  its 
unique existence  for the functions
$F\in L^{\infty}(0,T;L^2(\Omega))$ and with an initial condition of the type 
$\lim_{t\downarrow 0} \Vert u(\cdot,t)\Vert_* = 0$, where  
$\Vert \cdot\Vert_*$ is a certain  norm that is weaker than the $L^2$-norm.  A norm estimate for the week solution was however 
given via the norm of $F$ in $L^2(0,T;L^2(\Omega))$.
Thus the results presented in \cite{SY} show a certain inconsistency 
between the inclusion $F\in L^{\infty}(0,T;L^2(\Omega))$ and the solution
norm estimate via the norm of $F$ in $L^2(0,T;L^2(\Omega))$. 
In this paper, this inconsistency 
is resolved by a new definition of the weak solution that is based on the suggested definition of the Caputo derivative in the fractional Sobolev spaces. 

In this way, the maximum regularity of the fractional diffusion equation with the Caputo fractional derivative is shown in this paper. Let us mention that in \cite{Ba} the $W^{\alpha,p}(0,T)$-
regularity with $p > 1$ was proved for the fractional differential equations with the Riemann-Liouville time-fractional derivative.

The rest of this paper is organized in three sections.  In Section 2, 
the Riemann-Liouville fractional integral and the related Abel integral equations in the fractional Sobolev spaces are revisited. 
The result (Theorem \ref{t21}) of Section 2 forms a basis for investigation of the Caputo fractional derivative in the fractional Sobolev spaces 
in Section 3 where we establish the norm equivalence between 
the $L^2$-norm of $\ppp_t^{\alpha}u$ and the fractional Sobolev norm of
$u$ (Theorem \ref{t31}). In particular, we suggest a new interpretation of the Caputo derivative in the fractional Sobolev spaces and prove some important norm equivalences. Finally,  Section 4 is devoted to investigation of the maximum regularity of the solutions to some initial-boundary-value problems for the fractional diffusion equations with the Caputo time-derivative in the fractional Sobolev spaces. We introduce a notion of a week solution  to the problem under consideration and show both its uniqueness, existence, and the corresponding norm estimates. 

%\vspace{0.4cm}

%\noindent
\section{The Riemann-Liouville integral  in the fractional Sobolev spaces}

In this section, we first remind the reader of some known properties of the Riemann-Liouville fractional integral operator and then formulate and prove one of our main results. We start this section with some definitions of the operators 
and the functional spaces which we need in the further discussions. 
Throughout this paper, we always assume that $0 < \alpha < 1$ if we 
do not specify another condition.

The Riemann-Liouville fractional integral operator $J^{\alpha}: L^2(0,T) \to L^2(0,T)$ is defined by the formula (see e.g. \cite{GV})
$$
(J^{\alpha}y)(t) = \frac{1}{\Gamma(\alpha)}\int^t_0
(t-s)^{\alpha-1}y(s) ds, \quad 0\le t \le T, \quad 0 < \alpha \le 1,
\quad J^0 = I.
$$
By $L^2 := L^2(0,T)$ and $H^{\alpha}(0,T)$ we mean the usual $L^2$-space
and the fractional Sobolev space on the interval $(0,T)$ (see e.g. \cite{Ad}, Chapter 
VII), respectively. The 
$L^2$-norm and the scalar product in $L^2$ are denoted by  
$\Vert\cdot\Vert_{L^2}$ and $(\cdot,\cdot)_{L^2}$, respectively.  By $\sim$ we 
mean a norms equivalence.
We set 
$$
_0 H^{\alpha}(0,T) = \{ u \in H^{\alpha}(0,T): \thinspace
u(0) = 0\}
$$
if $\frac{1}{2} < \alpha \le 1$ and we identify 
$_0H^{\alpha}(0,T)$ with $H^{\alpha}(0,T)$ for $0 \le \alpha < 
\frac{1}{2}$.

For Hilbert spaces $X$ and $Y$ and an operator $K: X \to Y$
defined in $X$, by $\mathcal{D}(K)$ and $\mathcal{R}(K)$ we denote
the domain and the range of $K$, respectively.
It can be easily verified that the Riemann-Liouville operator $J^{\alpha}: L^2 \to L^2$ is injective (Theorem 5.1.1 in \cite{GV}).  
Therefore there exists an inverse operator to the Riemann-Liouville operator $J^{\alpha}$ and we denote it  
by $J^{-\alpha}$.  By the definition, 
$\mathcal{D}(J^{-\alpha}) = \mathcal{R}(J^{\alpha})$. 
To deal with the operator $J^{-\alpha}$ in $\mathcal{D}(J^{-\alpha})$, we thus have to describe the range of the Riemann-Liouville operator $J^{\alpha}$
incorporated with the norm, 
which is given in the following theorem (for $0 \le \alpha < \frac{1}{2}$,
a part of our results is already formulated and proved in Theorem 18.3 
from \cite{SKM}).
%\\
\begin{thm}%{\bf Theorem 2.1}.\\
\label{t21}
$\mbox{ }$
\\
(i) 
\begin{align*}
&\Vert J^{\alpha}u\Vert_{H^{\alpha}(0,T)} \sim \Vert u\Vert_{L^2},
\quad u \in L^2(0,T),\\
&\Vert J^{-\alpha}v\Vert_{L^2} \sim \Vert v\Vert_{H^{\alpha}(0,T)},
\quad v \in \mathcal{R}(J^{\alpha}).
\end{align*}
\\
(ii) 
$$
\mathcal{R}(J^{\alpha}) =
\left\{
\begin{array}{rl}
&H^{\alpha}(0,T), \quad 0 \le \alpha < \frac{1}{2}, \\
&_{0}H^{\alpha}(0,T), \quad  \frac{1}{2} < \alpha \le 1,\\
&\left\{ u \in H^{\frac{1}{2}}(0,T):\thinspace
\int^T_0 t^{-1}\vert u(t)\vert^2 dt < \infty\right\},
\quad \alpha = \frac{1}{2}.\\ 
\end{array}\right.
$$
\end{thm}
For the proof of Theorem \ref{t21}, some auxiliary statements that were derived 
in \cite{GY} are needed. For the sake of completeness, we give here both the
formulations and the proofs of these results. 

By  $J = J^1$ we denote the integral $(Jy)(t) = \int^t_0 y(s) ds$ for 
$0\le t \le T$ and by $I: L^2 \to L^2$ the identity mapping. 
In this section, we consider the space $L^2(0,T)$ over $\C$ with the 
scalar product $(u,v)_{L^2} = \int^T_0 u(t)\overline{v(t)} dt$, and
$\Re \eta$ and $\Im \eta$ denote the real and the imaginary parts of 
a complex number $\eta$, respectively .

%{\bf Lemma 2.1}.\\
\begin{lem}
\label{l21}
For any $u \in L^2$, the inequality $\Re (Ju,u)_{L^2} \ge 0$ holds true and 
$\mathcal{R}(I+J) = L^2$.
\end{lem}
%{\bf Proof.}
%\\
\begin{proof}
First the notations  $\Re Ju(t) = \va(t)$ and $\Im Ju(t) = \psi(t)$ 
are introduced. With these notations, the following chain of equalities and 
inequalities can be easily obtained: 
\begin{align*}
& \Re \thinspace (Ju,u)_{L^2}
= \int^T_0 \left(\int^t_0 u(s)ds\right) \overline{u}(t) dt  
= \int^T_0 Ju(t)\frac{d}{dt}\overline{Ju(t)} dt\\
=& \int^T_0 \left(\va(t)\frac{d\va}{dt} + \psi(t)\frac{d\psi}{dt}
\right) dt
= \frac{1}{2}(\va(t)^2 + \psi(t)^2)\vert^{t=T}_{t=0}
= \frac{1}{2}\vert Ju(T)\vert^2 \ge 0.
\end{align*}
Therefore $\Re (Ju,u)_{L^2} \ge 0$ for $u \in L^2$.
Next we have
$$
(\lambda I+J)^{-1}u(t)
 = \lambda^{-1}u(t) - \lambda^{-2}\int^t_0 e^{-(t-s)/\lambda}u(s) ds, 
\quad 0\le t \le T.                  \eqno{(2.1)}
$$
Setting $\lambda = 1$, by (2.1) the operator $(I+J)^{-1}u$ is defined for all
$u \in L^2$, which implies that $\mathcal{R}(I+J) = L^2$.
The proof is completed.
\end{proof}
%\\

It follows from Lemma \ref{l21} that $J$ is a maximal accretive operator 
(Chapter 2, \S1 in \cite{Ta}) and thus the 
assumption 6.1 in Chapter 2, \S6 of \cite{Pa} holds valid, so that for 
$0 < \alpha < 1$ one can 
define the fractional power $J(\alpha)$ of the integral operator $J$ by the 
formula
$$
J(\alpha)u = \frac{\sin\pi\alpha}{\pi}\int^{\infty}_0
\lambda^{\alpha-1}(\lambda I + J)^{-1}Ju\, d\lambda, \quad 
u \in \mathcal{D}(J) = L^2
$$
(see also Chapter 2, \S3 in \cite{Ta}).
It turns out that the  fractional power $J(\alpha)$ of the integral 
operator $J$ coincides with 
the Riemann-Liouville fractional integral operator on $L^2$ as stated 
and proved below. 
%\\
%{\bf Lemma 2.2.}\\
\begin{lem}
\label{l22}
$$
(J(\alpha)u)(t) = (J^{\alpha}u)(t), \qquad 0\le t \le T,
\quad u \in L^2, \quad 0 < \alpha < 1.
$$
\end{lem}
%{\bf Proof.}
\begin{proof}
By (2.1), we have
$$
(\lambda I + J)^{-1}Ju(t) 
= \lambda^{-1}\int^t_0 e^{-(t-s)/\lambda}u(s) ds,
$$
and by the change of the variables $\eta = \frac{t-s}{\lambda}$, we
obtain
\begin{align*}
& \frac{\sin\pi\alpha}{\pi}\int^{\infty}_0 \lambda^{\alpha-1}
(\lambda I + J)^{-1}Ju(t)\, d\lambda
= \frac{\sin\pi\alpha}{\pi}\int^{\infty}_0 \lambda^{\alpha-2}
\left( \int^t_0 e^{-(t-s)/\lambda} u(s)ds \right)d\lambda\\
=& \frac{\sin\pi\alpha}{\pi}\int^t_0 u(s) \left(
\int^{\infty}_0 \lambda^{\alpha-2} e^{-(t-s)/\lambda} d\lambda
\right)ds 
= \frac{\sin\pi\alpha}{\pi}\int^t_0 u(s)
\left( \int^{\infty}_0 \eta^{-\alpha} e^{-\eta} d\eta\right)
(t-s)^{\alpha-1} ds\\
= & \frac{\Gamma(1-\alpha)\sin\pi\alpha}{\pi}
\int^t_0 u(s)(t-s)^{\alpha-1} ds.
\end{align*}
Now the known formula $\Gamma(1-\alpha)\Gamma(\alpha) = \frac{\pi}{\sin\pi\alpha}$
implies the statement of the lemma.
\end{proof}

Next we consider the differential operator 
$$
\left\{
\begin{array}{rl}
&(Au)(t) = -\frac{d^2u(t)}{dt^2}, \quad 0 < t < T,\\
&\mathcal{D}(A) = \{u \in H^2(0,T):\thinspace u(0) = \frac{du}{dt}(T) = 0\}.
\end{array}\right.
\eqno{(2.2)}
$$
Note that the boundary conditions $u(0) = \frac{du}{dt}(T) = 0$ should be 
interpreted as the traces of $u$ in the Sobolev space $H^2(0,T)$ 
(see e.g., \cite{Ad} or \cite{LM}).
It is possible to define the fractional power $A^{\frac{\alpha}{2}}$ of 
the differential operator $A$ for 
$0 \le \alpha \le 1$ in terms of
the eigenvalues and the eigenfunctions of the eigenvalue problem for the 
operator $A$.  More precisely,
let $0 < \lambda_1 < \lambda_2 < \cdots $ be the eigenvalues and 
$\psi_k$, $k \in \N$  the corresponding normed eigenfunctions of $A$.  
It is easy to derive the explicit 
formulas for $\lambda_k,\ \psi_k,\ k \in \N$, namely,  
$\lambda_k = \frac{(2k-1)^2\pi^2}{4T^2}$ and 
$\psi_k(t) = \frac{\sqrt{2}}{\sqrt{T}} \sin \sqrt{\lambda_k}t$.  
In particular, we note that $\psi_k(0) = 0$ and
$\psi_k \in H^2(0,T)$.  It is known that $\{\psi_k\}_{k\in\N}$ is 
an orthonormal basis of $L^2$.  Then the fractional power 
$A^{\frac{\alpha}{2}},\ 0\le \alpha \le 1$  of 
the differential operator $A$ is defined by the relations 
$$
\left\{
\begin{array}{rl}
& A^{\frac{\alpha}{2}}u = \sum_{k=1}^{\infty} \lambda_k^{\frac{\alpha}{2}}
(u,\psi_k)_{L^2}\psi_k, \quad u \in \mathcal{D}(A^{\frac{\alpha}{2}}),\\
& \mathcal{D}(A^{\frac{\alpha}{2}}) 
= \{u \in L^2:\thinspace
\sum_{k=1}^{\infty} \lambda_k^{\alpha}\vert (u,\psi_k)_{L^2}\vert^2
< \infty\}, \\
& \Vert u\Vert_{\mathcal{D}(A^{\frac{\alpha}{2}})} 
= \left( \sum_{k=1}^{\infty} \lambda_k^{\alpha}\vert (u,\psi_k)_{L^2}\vert^2
\right)^{\frac{1}{2}}.
\end{array}\right.
\eqno{(2.3)}
$$
According to \cite{F} (see also Lemma 8 in \cite{GY}), the domain 
$\mathcal{D}(A^{\frac{\alpha}{2}})$ can be described as follows: 
$$
\mathcal{D}(A^{\frac{\alpha}{2}})=
\left\{
\begin{array}{rl}
&H^{\alpha}(0,T), \quad  0\le \alpha < \frac{1}{2}, \\
&_{0}H^{\alpha}(0,T), \quad  \frac{1}{2} < \alpha \le 1,\\
&\left\{ u \in H^{\frac{1}{2}}(0,T):\thinspace
\int^T_0 t^{-1}\vert u(t)\vert^2 dt < \infty\right\}, \quad
\alpha=\frac{1}{2}. \\ 
\end{array}\right.
                                                 \eqno{(2.4)}
$$
The relation (2.4) holds not only algebraically but also topologically, 
that is,
$$
\Vert A^{\frac{\alpha}{2}}u\Vert_{L^2} \sim \Vert u\Vert_{H^{\alpha}(0,T)},
\quad 0\le \alpha \le 1, \thinspace u \in \mathcal{D}(A^{\frac{\alpha}{2}}).
                                             \eqno{(2.5)}
$$  
In particular, the inclusion $\mathcal{D}(A^{\frac{\alpha}{2}}) 
\subset H^{\alpha}(0,T)$ 
holds true. 
We note that in \cite{GY} the case of the operator $A=\frac{d^2}{d\, t^2}$ 
with the domain 
$\left\{u \in H^2(0,1):
\thinspace \frac{du}{dt}(0) = u(1) = 0\right\}$ was considered, 
which is reduced  
to our case by a simple change of the variables.
 
Now we are ready to prove Theorem \ref{t21}.
\begin{proof}
First of all, it can be directly verified that  $\mathcal{D}(J^{-1}) 
= J(L^2) = {_{0}H^1(0,T)}$, $(J^{-1}w)(t) = \frac{dw(t)}{dt}$, and 
$$
\Vert J^{-1}v\Vert_{L^2} = \Vert v\Vert_{H^1(0,T)}, \quad
v \in {_{0}H^1(0,T)}.                            
$$
Therefore by (2.5) we obtain the norm equivalence 
$$
\Vert J^{-1}v\Vert_{L^2} \sim \Vert A^{\frac{1}{2}}v\Vert_{L^2},
\quad v \in {_{0}H^1(0,T)} =\mathcal{D}(J^{-1}) 
= \mathcal{D}(A^{\frac{1}{2}}).                   
$$
Direct calculations show that both $J^{-1}$ and $A^{\frac{1}{2}}$ are 
maximal accretive in $L^2$.  Hence the Heinz-Kato inequality (see e.g.
Theorem 2.3.4 in \cite{Ta}) yields
$$
\Vert J^{-\alpha}v\Vert_{L^2} \sim \Vert A^{\frac{\alpha}{2}}v\Vert_{L^2},
\quad v \in \mathcal{D}(A^{\frac{\alpha}{2}}), \quad
\mathcal{D}(J^{-\alpha}) = \mathcal{D}(A^{\frac{\alpha}{2}}).
$$
By (2.5) the norm equivalence $\Vert J^{-\alpha}v\Vert_{L^2} \sim \Vert v\Vert
_{H^{\alpha}(0,T)}$ holds true for $v \in \mathcal{D}(J^{-\alpha}) 
= \mathcal{R}(J^{\alpha})$.  Next, setting $v = J^{\alpha}u 
\in \mathcal{D}(J^{-\alpha})$ with
any $u \in L^2$, by (2.5) we obtain the following norm equivalences
$$
\Vert u\Vert_{L^2} \sim \Vert A^{\frac{\alpha}{2}}(J^{\alpha}u)
\Vert_{L^2} \sim \Vert J^{\alpha}u\Vert_{H^{\alpha}(0,T)}.
$$
Thus the proof of Theorem \ref{t21} (i) is completed.

Theorem \ref{t21} (ii) follows from the relation (2.4) and the equality $\mathcal{D}(J^{-\alpha}) 
= \mathcal{R}(J^{\alpha})$. 

\end{proof}

%\vspace{0.4cm}

%\noindent
\section{The Caputo derivative in the fractional Sobolev spaces}

\noindent
The main aim of this section is to define the Caputo derivative in the 
fractional Sobolev spaces and to prove the norm equivalence (Theorem \ref{t31}).  
The original definition of the Caputo derivative is by  the formula
$$
\ppp_t^{\alpha}u(x,t) = \frac{1}{\Gamma(1-\alpha)}\int^t_0
(t-s)^{-\alpha}\frac{\ppp u}{\ppp s}(x,s) ds, \quad 0\le t \le T, 
\quad 0 < \alpha < 1.
$$
Thus $\ppp_t^{\alpha}u$ is defined pointwise for $x \in \Omega$ for 
$u \in H^1(0,T)$ and the definition requires some suitable 
conditions on the first order derivative  $\frac{\ppp u}{\ppp s}$.  
On the other hand, since  $\ppp_t^{\alpha}u$ is the $\alpha$-th derivative 
with $0<\alpha <1$, one
can expect a natural interpretation of $\ppp_t^{\alpha}u$ for the functions 
from the fractional Sobolev space
$H^{\alpha}(0,T)$.  Suggesting this interpretation is the main purpose of 
this section.

We start with introducing a linear span $W$  of the  eigenfunctions 
$\psi_k,\ k\in \N$ of the differential operator $A$ that is defined by the 
formula (2.2):
$$
W = \left\{\sum_{k=1}^N a_k\psi_k:\thinspace N \in \N, \thinspace 
a_k \in \R\right\}.
$$
For the functions from $W$, the following result holds true:
%\\
%{\bf Lemma 3.1}.\\
\begin{lem}
\label{l31}
$$
\ppp_t^{\alpha}\va(t) = J^{-\alpha}\va(t), \quad \va \in W, \ \ 
0 < t < T.
$$
\end{lem}
\begin{proof}
For any $\va \in W \subset H^2(0,T)$,  the Riemann-Liouville fractional 
derivative is defined by the relation
$$
D_t^{\alpha}\va(t) = \frac{1}{\Gamma(1-\alpha)}\frac{d}{dt}\int^t_0
(t-s)^{-\alpha}y(s) ds, \quad 0\le t \le T.
$$
Then the relation
$$
D_t^{\alpha}\va(t) = \frac{\va(0)t^{-\alpha}}{\Gamma(1-\alpha)}
+ \ppp_t^{\alpha}\va(t), \quad 0<t<T, \quad \va \in W
                                               \eqno{(3.1)}
$$
between the Caputo and the Riemann-Liouville fractional 
derivatives holds true. 
Indeed, integration by parts implies
$$
\int^t_0 (t-s)^{-\alpha}\va(s)ds
= \left[ \va(s)\frac{(t-s)^{1-\alpha}}{1-\alpha}\right]^{s=0}_{s=t}
+ \frac{1}{1-\alpha}\int^t_0 (t-s)^{1-\alpha}\frac{d\va}{ds}(s)ds.
$$
Applying the differentiation operator $\frac{d}{dt}$ to both sides of 
the last formula and noting the inclusion $\va \in H^2(0,T)$, we arrive at 
the formula (3.1).

Since $\va \in W \subset \mathcal{D}(J^{-\alpha})
= \mathcal{R}(J^{\alpha})$, by Theorem \ref{t21} (ii) there exists
$\www\va \in L^2$ such that $\va = J^{\alpha}\www\va$.  On the other hand,
the Riemann-Liouville fractional derivative is the 
left inverse operator to the Riemann-Liouville fractional integral and 
we have 
$$
D_t^{\alpha}J^{\alpha}\www\va = \www\va         \eqno{(3.2)}
$$
for $\www\va \in L^1(0,T)$ (see e.g. Theorem 6.1.2 in 
\cite{GV}).  Hence
$D_t^{\alpha}\va = \www\va$.  For $\va \in W$, the condition $\va(0) = 0$ 
is fulfilled.  Thus (3.1) yields 
$$
D_t^{\alpha}\va(t) = \frac{\va(0)t^{-\alpha}}{\Gamma(1-\alpha)}
+ \ppp_t^{\alpha}\va(t) = \ppp_t^{\alpha}\va, \quad \va\in W.
$$
Therefore $\ppp_t^{\alpha}\va = \www\va 
= J^{-\alpha}\va$.  Thus the proof of the lemma is completed.
\end{proof}

Lemma \ref{l31} still provides just a pointwise definition of the Caputo fractional 
derivative $\ppp_t^{\alpha}\va$
for a function $\va \in W$.  Now let us consider the closure of the operator 
$\ppp_t^{\alpha}$ in $W$ 
to define it over the whole space $\mathcal{R}(J^{\alpha})$.
For $\va \in W$, the inequality
$$
\Vert \ppp_t^{\alpha}\va\Vert_{L^2} 
= \Vert J^{-\alpha}\va\Vert_{L^2}
\le C\Vert \va\Vert_{H^{\alpha}(0,T)}
$$ 
is guaranteed by Theorem \ref{t21}.  Therefore the linear operator
$\va \mapsto \ppp_t^{\alpha}\va$ is bounded from $W \subset 
\mathcal{D}(A^{\frac{\alpha}{2}})$ to 
$L^2$.  Since $W$ is dense in $\mathcal{R}(J^{\alpha}) = 
\mathcal{D}(A^{\frac{\alpha}{2}})$, 
the operator $\ppp_t^{\alpha}$ can be uniquely extended from $W$ to 
the domain $\mathcal{R}(J^{\alpha})$.  
This extended operator is defined on the whole space $\mathcal{R}(J^{\alpha})$ 
and is bounded from $\mathcal{R}(J^{\alpha})$ to $L^2$.  
For the extension  of $\ppp_t^{\alpha}$, the same notation as for the 
pointwise Caputo fractional derivative will be used in the rest of the paper.
In other words, let $u \in \mathcal{R}(J^{\alpha}) \subset H^{\alpha}(0,T)$ 
and a sequence $\va_n \in W$ converge to $u$ in $\mathcal{R}(J^{\alpha})$.  
Then $\ppp_t^{\alpha}u$ will be interpreted as follows: 
$$
\ppp_t^{\alpha}u(t) = \lim_{n\to\infty}
\left( \frac{1}{\Gamma(1-\alpha)}\int^t_0
(t-s)^{-\alpha}\frac{d \va_n}{ds}(s) ds\right) \quad \mbox{in $L^2$}.   
                                               \eqno{(3.3)}
$$ 
Let us note that this definition is correct, i.e., 
it is independent of the choice of the sequence  $\va_n$.
Indeed, for a function $u \in \mathcal{R}(J^{\alpha})$ 
let $\{\va_n\}_{n\in \N}$ and 
$\{\www \va_n\}_{n\in \N} $ be two sequences that approximate the function 
$u$: $\va_n \to u$ and $\www\va_n \to u$ in $\mathcal{R}(J^{\alpha})$ as 
$n \to \infty$.  Then $\lim_{n\to\infty}
\Vert \va_n - \www\va_n\Vert_{H^{\alpha}(0,T)}
= 0$ and the boundedness of the operator $\ppp_t^{\alpha}: W \subset
H^{\alpha}(0,T) \to L^2$ yields 
$\lim_{n\to\infty} \Vert\ppp_t^{\alpha}
(\va_n - \www\va_n)\Vert_{L^2} = 0$, that is,
$\lim_{n\to\infty} \ppp_t^{\alpha}\va_n
= \lim_{n\to\infty} \ppp_t^{\alpha}\www\va_n$ in $L^2$, 
what we wanted to show. 

In what follows, the Caputo fractional derivative 
$\ppp_t^{\alpha}u$ for $u \in \mathcal{R}(J^{\alpha})$ will be defined by 
(3.3) and no more pointwise.

Let us now derive some properties of the Caputo fractional derivative 
$\ppp_t^{\alpha}u$ defined by (3.3).
\begin{thm}
\label{t31}
%{\bf Theorem 3.1.}
$\mbox{ }$
\\
(i) $\ppp_t^{\alpha}u = J^{-\alpha}u$ in $L^2$ for $u 
\in \mathcal{R}(J^{\alpha})$.
\\
(ii) $\Vert \ppp_t^{\alpha}u\Vert_{L^2} \sim \Vert u\Vert_{H^{\alpha}(0,T)}$
for $u \in \mathcal{R}(J^{\alpha})$.
\end{thm}
\begin{proof}
The part (i) of the theorem follows directly from the definition 
(3.3).  As to the part (ii), 
by Theorems \ref{t31} (i) and \ref{t21}, for $u \in \mathcal{R}(J^{\alpha})$ we have the 
norm equivalence 
$$
\Vert \ppp_t^{\alpha}u\Vert_{L^2} 
= \Vert J^{-\alpha}u\Vert_{L^2} \sim \Vert u\Vert_{H^{\alpha}(0,T)},
$$
which completes the proof.
\end{proof}
%\vspace{0.2cm}
%\\
Before we start with analysis of the fractional diffusion equation in the 
fractional Sobolev spaces, let us mention that by  Theorem \ref{t31} (i), 
the solution $u$ to the equation 
$$
\ppp_t^{\alpha}u = f, \quad f\in L^2
$$
is given by $u = J^{-\alpha}f$ and 
$$
u \in 
\left\{
\begin{array}{rl}
&H^{\alpha}(0,T), \quad 0 \le \alpha < \frac{1}{2}, \\
&_{0}H^{\alpha}(0,T), \quad  \frac{1}{2} < \alpha \le 1,\\
&\left\{ u \in H^{\frac{1}{2}}(0,T):\thinspace
\int^T_0 t^{-1}\vert u(t)\vert^2 dt < \infty\right\},
\quad \alpha = \frac{1}{2}.\\ 
\end{array}\right.
$$
This result is well-known 
(see e.g. \cite{LuG}).  Our Theorem \ref{t31} asserts not only 
this formula but also a characterization of the range 
$\mathcal{R}(J^{\alpha})$ in the framework of the extended definition (3.3) of 
the fractional Caputo derivative.

Since $u \in \mathcal{R}(J^{\alpha})$ 
implies $u(0) = 0$ for $\frac{1}{2} < \alpha < 1$, the Caputo and the 
Riemann-Liouville fractional derivatives coincide in $L^2$, that is,
the formula $\ppp_t^{\alpha}u = D_t^{\alpha}u$ holds true, which  
corresponds to the relation (3.1) for a wider class of the functions 
compared to $W$.  
For $0 < \alpha < \frac{1}{2}$, we have $\mathcal{R}(J^{\alpha})
= H^{\alpha}(0,T)$, and with our extended definition of $\ppp_t^{\alpha}$ 
given by (3.3) and a suitably extended definition of $D_t^{\alpha}$ in 
$H^{\alpha}(0,T)$, the relation 
$\ppp_t^{\alpha}u = D_t^{\alpha}u$ is true for $u \in  
\mathcal{R}(J^{\alpha}) = H^{\alpha}(0,T)$, too.  
We note that for $0<\alpha < \frac{1}{2}$, the pointwise definition of 
the Caputo fractional derivative $\ppp_t^{\alpha}u$ involves the first order 
derivative $\frac{du}{dt}$ and thus it does not make any sense for the functions 
$u \in H^{\alpha}(0,T)$.

%\vspace{0.4cm}

%\noindent
\section{Maximal regularity of solutions to the fractional diffusion equation 
with the Caputo derivative}

Let $\Omega \subset \R^n$ be a bounded domain with the smooth boundary 
$\ppp\Omega$, and let $(u,v)_{L^2(\Omega)}$ be the scalar product in 
$L^2(\Omega)$, that is,
$(u,v)_{L^2(\Omega)} = \int_{\Omega} u(x)v(x) dx$.

In the first part of this section, we deal with the following 
initial-boundary-value problem for the fractional diffusion equation with the 
Caputo time-fractional derivative:
$$
\left\{
\begin{array}{rl}
&\ppp_t^{\alpha}u(x,t) = -Lu(x,t) +  F(x,t), \quad x\in \Omega, 
\quad 0 < t < T,\\
&u(x,0) = 0, \quad x \in \Omega,\\
&u\vert_{\ppp\Omega\times (0,T)} = 0,
\end{array}\right. 
\eqno{(4.1)}
$$
where $L$ is a symmetric uniformly elliptic operator in the form 
$$
Lu(x,t) = -\sum_{j,k=1}^n \frac{\partial }{\partial x_j}\left( 
a_{jk}(x)\frac{\partial }{\partial x_k} u(x,t)\right) - c(x)u(x,t)
$$
and the conditions 
$$
\left\{
\begin{array}{rl}
&a_{jk} = a_{kj} \in C^1(\overline{\Omega}), \quad j,k=1,..., n,\\
&\mbox{there exists a constant } \nu_0 > 0 \mbox{ such that}\\ 
&\sum^n_{j,k=1} a_{jk}(x)\xi_j\xi_k \ge \nu_0\sum_{j=1}^n \xi_i^2,
\quad x\in \overline{\Omega}, \thinspace \xi_1, ...., \xi_n \in \R 
\end{array}\right.
\eqno{(4.2)}
$$
are fulfilled. 
Moreover, we assume that $c(x) \le 0,\ x \in C(\overline{\Omega})$ 
(in the second part of the section this condition will be removed). 
As to the operator $L$, our method will be applied  for a more general 
elliptic operator in the second part of this section, 
but first we restrict ourselves to a self-adjoint and positive-definite 
operator for the sake of simplicity.

In this paper, we are interested in a weak solution to the problem (4.1) 
that is defined on the basis of Theorem \ref{t31}.
\begin{defn}[Definition of a weak solution]
Let $F \in L^2(0,T; L^2(\Omega))$.  We call $u$ a weak solution to the problem 
(4.1) if $u \in L^2(0,T;H^2(\Omega) \cap H^1_0(\Omega))$ and the following 
conditions are satisfied 
\\
(i) 
$$ 
J^{-\alpha}u \in L^2(0,T; L^2(\Omega)),
$$
(ii) 
$$
\ppp_t^{\alpha}u(x,t) = -Lu(x,t) + F(x,t) \quad \mbox{in } 
L^2(0,T;L^2(\Omega)).     
                                                        \eqno{(4.3)}
$$
\end{defn}
We note that the inclusion $J^{-\alpha}u \in L^2(0,T;L^2(\Omega))$ implies 
$u(x,\cdot) \in \mathcal{R}(J^{\alpha})$ for almost all $x \in \Omega$.
Hence, in the case $\frac{1}{2} < \alpha < 1$ it follows from Theorem \ref{t21} (ii) 
that
$u(x,\cdot) \in {_{0}H^{\alpha}(0,T)}$ and so $u(x,0) = 0$ for almost all
$x \in \Omega$.  In the case $\alpha = \frac{1}{2}$, the condition
$\int^T_0 \frac{\vert u(x,t)\vert^2}{t} dt < \infty$ holds true for 
$u(x,\cdot) \in \mathcal{R}(J^{\alpha})$.  This condition implicitly describes 
the behavior of the function $u$ in a small neighborhood of the point $t=0$, 
but one cannot conclude from this condition that $u(x,0) = 0$.  
In the case $0 < \alpha < \frac{1}{2}$, 
the initial condition $u(x,0) = 0$ of the problem (4.1) is not meaningful 
at all, because a function $\eta \in H^{\alpha}(0,T)$ has no
trace at $t=0$ if $0 < \alpha < \frac{1}{2}$ in general.
Thus an initial condition of the problem (4.1) has to be posed  depending on 
the trace of the functions from $H^{\alpha}(0,T)$ at the point $t=0$. 
To illustrate the remarks given above, let us consider a simple example. 
\begin{exa}
In the problem (4.1), we set  $n=1$, $\Omega = (0,1)$, 
$0 < \alpha < \frac{1}{2}$, and  
$$
F(x,t) = \frac{-2\Gamma\left(\delta + \frac{1}{2}\right)}
{\Gamma\left(\alpha+\delta+\frac{1}{2}\right)}t^{\alpha+\delta-\frac{1}{2}}
+ x(x-1)t^{\delta-\frac{1}{2}},
$$
where $\delta > 0$ and $\alpha + \delta - \frac{1}{2} < 0$.
Then the inclusion $F \in L^2(0,T;L^2(0,1))$ holds true and we can 
directly check that the function 
$$
u(x,t) = \frac{\Gamma\left(\delta + \frac{1}{2}\right)}
{\Gamma\left(\alpha+\delta+\frac{1}{2}\right)}x(x-1)
t^{\alpha+\delta-\frac{1}{2}}
$$
satisfies the fractional diffusion equation 
$$ 
\ppp_t^{\alpha}u(x,t) = \ppp_x^2 u(x,t) + F(x,t),\ 0 < x < 1,\ t > 0.
$$
Moreover, $u \in L^2(0,T;H^2(0,1) \cap H^1_0(0,1))$ and 
$\ppp_t^{\alpha}u = x(x-1)t^{\delta-\frac{1}{2}}
\in L^2(0,T; L^2(0,1))$  
and thus $u$ is a week solution to the problem (4.1). 
However, since $\alpha + \delta-\frac{1}{2} < 0$, the value of 
$u(x,0)$ in the sense of  $\lim_{t\to 0} \Vert u(\cdot,t)\Vert_{L^2(0,1)}$ 
does not exist.  
Thus we see that there exists a solution to the problem (4.1) that does not 
admit the initial condition. 
\end{exa}

Now we state and prove our main result regarding the weak solution to the 
problem (4.1).
%{\bf Theorem 4.1.}\\
\begin{thm}
\label{t41}
Let $F \in L^2(0,T;L^2(\Omega))$ and $0 < \alpha < 1$.
Then there exists a unique weak solution to the problem (4.1).
Moreover there exists a constant $C>0$ such that the norm estimate 
$$
\Vert u\Vert_{H^{\alpha}(0,T;L^2(\Omega))}
+ \Vert u\Vert_{L^2(0,T;H^2(\Omega))} \le
C\Vert F\Vert_{L^2(0,T;L^2(\Omega))}    \eqno{(4.4)}
$$
holds true for all $F \in L^2(0,T;L^2(\Omega))$.
\end{thm}
\begin{proof}
The operator $L$ defined by 
$$
(L v)(x) = -\sum_{j,k=1}^n \frac{\partial }{\partial x_j}\left( 
a_{jk}(x)\frac{\partial }{\partial x_k} v(x)\right) - c(x)v(x), 
\quad v \in \mathcal{D}(L) := H^2(\Omega) \cap H^1_0(\Omega)
                                     \eqno{(4.5)}
$$
is a  positive-definite and self-adjoint operator in $L^2(\Omega)$.  
Let $0 < \mu_1 \le \mu_2 \le \cdots  $ be all the 
eigenvalues of $L$, where $\mu_k$ appears in the sequence as often 
as its multiplicity requires.  Let $\va_k, k\in \N$ be the eigenfunction of
$L$ corresponding to the eigenvalue $\mu_k$.  It is known that $\mu_k 
\to \infty$ as $k \to \infty$ and the eigenfunctions $\va_k$ can be chosen to 
be orthonormal, i.e., 
$(\va_j,\va_k)_{L^2(\Omega)} = 1$ if $j=k$ and $(\va_j,\va_k)
_{L^2(\Omega)} = 0$ if $j \ne k$.
These eigenfunctions $\{\va_k\}_{k\in \N}$ build an orthonormal basis
of $L^2(\Omega)$.
\\
{\bf (i) Uniqueness of the weak solution.}
Let $w$ be a weak solution to (4.1) with $F=0$.  Since
$\ppp_t^{\alpha}w$, $Lw \in L^2(0,T;L^2(\Omega))$ and 
$L\va_k = \mu_k\va_k$, the functions $w_k(t) := (w(\cdot,t), \va_k)
_{L^2(\Omega)},\ k \in \N$ belong to the space $\mathcal{R}(J^{\alpha})$ and 
satisfy the relation $\ppp_t^{\alpha}w_k(t) = -\mu_kw_k(t)$. 
Therefore by Theorem \ref{t31} (i), we have
$J^{-\alpha}w_k = -\mu_kw_k$ in $L^2$, that is, $w_k = -\mu_kJ^{\alpha}w_k$
in $L^2$:
$$
w_k(t) = \frac{-\mu_k}{\Gamma(\alpha)}\int^t_0 (t-s)^{\alpha-1}
w_k(s) ds, \quad 0 < t < T.
$$
Hence
$$
\vert w_k(t)\vert \le C\int^t_0 (t-s)^{\alpha-1}\vert w_k(s)\vert ds, 
\quad 0 < t < T.               \eqno{(4.6)}
$$
The generalized Gronwall inequality (see e.g., Theorem 7.1.2  in 
 \cite{H}) yields then the relations $w_k(t) = 0,\ k \in \N$ for $0 < t < T$.
Since $\{\va_k\}_{k\in \N}$ is an orthonormal basis of $L^2(\Omega)$, 
we obtain the relation
$w(\cdot,t) =0$, $0 < t < T$ and thus 
the uniqueness of the weak solution is proved.
\\
{\bf (ii) Existence of the weak solution.}
Our construction of the weak solution follows the lines of the one presented 
in Theorem 2.2 of \cite{SY}.
Let us introduce the sequence $u_k(t) := (u(\cdot,t),\va_k)_{L^2(\Omega)}, 
k \in \N$ and construct a candidate for the weak solution in the form
$$
\widetilde{u}(x,t) = \sum_{k=1}^{\infty} p_k(t)\va_k(x), \quad 
x\in \Omega, \thinspace 0 < t < T     \eqno{(4.7)}
$$
with the functions $p_k$ given by the formula
$$
p_k(t) = \int^t_0 (F(\cdot,s),\va_k)_{L^2(\Omega)}
(t-s)^{\alpha-1}E_{\alpha,\alpha}(-\mu_k(t-s)^{\alpha}) ds.
$$
Here and henceforth $E_{\alpha,\beta}$ denotes 
the two-parameters Mittag-Leffler function defined by the series
$$
E_{\alpha,\beta}(z) = \sum_{k=0}^{\infty} \frac{z^k}{\Gamma(\alpha\,k+\beta)},
\quad z \in \C,\ \alpha >0.
$$
The function $E_{\alpha,\alpha}(z)$ is known to be an entire 
function.  Applying the same technique as the one used  in the proof of 
Theorem 2.2 in 
\cite{SY}, we can show that the series in (4.7) is convergent in 
$L^2(0,T;H^2(\Omega) \cap H^1_0(\Omega))$ and the norm estimate 
$$
\Vert \widetilde{u}\Vert_{L^2(0,T;H^2(\Omega))} \le 
C\Vert F\Vert_{L^2(0,T;L^2(\Omega))}           \eqno{(4.8)}
$$
as well as the relation
$$
\ppp_t^{\alpha}p_k(t) = -\mu_kp_k(t) 
+ (F(\cdot,t),\va_k)_{L^2(\Omega)}, \quad 0<t < T, \thinspace k\in \N
$$
hold true. 
Setting $\widetilde{u}_N(x,t) = \sum_{k=1}^N p_k(t)\va_k(x)$, 
we have the formula
\begin{align*}
&\ppp_t^{\alpha}\widetilde{u}_N = \sum^N_{k=1} p_k(t)(-\mu_k\va_k)
+ \sum^N_{k=1} (F(\cdot,t),\va_k)_{L^2(\Omega)}\va_k\\
=& -L\widetilde{u}_N + \sum_{k=1}^N (F(\cdot,t),\va_k)_{L^2(\Omega)}\va_k.
\end{align*}
Since $L\widetilde{u}_N \to L\widetilde{u}$ in $L^2(\Omega)$ as
$N \to \infty$ with a function $\widetilde{u} \in L^2(0,T;H^2(\Omega))$, 
we can derive the relation
$\lim_{N\to\infty} \ppp_t^{\alpha}\widetilde{u}_N = -L\widetilde{u}
+ F$ in $L^2(\Omega)$.  Therefore, $\ppp_t^{\alpha}\widetilde{u}
= -L\widetilde{u} + F$ and $\widetilde{u}(x,\cdot) \in \mathcal{R}
(J^{\alpha})$ for almost all $x \in \Omega$.  Thus $\widetilde{u}$ is the
weak solution and so $u = \widetilde{u}$.  
Since $-L$ is an elliptic operator of the second order, the norm estimate 
$\Vert v\Vert_{H^2(\Omega)} \le C\Vert Lv\Vert_{\LLL}$ is valid for a function 
$v \in H^2(\Omega) \cap H^1_0(\Omega)$ (see e.g. \cite{LM}).
Hence (4.8) yields the inequality 
$$
\Vert u\Vert_{L^2(0,T;H^2(\Omega))} \le C\Vert F\Vert_{L^2(0,T;L^2(\Omega))}.
                                                 \eqno{(4.9)}
$$
Moreover, Theorem \ref{t31} and the inequality (4.9) imply the following norm 
estimate
\begin{align*}
& \Vert u\Vert_{H^{\alpha}(0,T;L^2(\Omega))}
\le C\Vert \ppp_t^{\alpha}u\Vert_{L^2(0,T;L^2(\Omega))}
= C\Vert -Lu+F\Vert_{L^2(0,T;L^2(\Omega))}\\
\le & C\Vert F\Vert_{L^2(0,T;L^2(\Omega))}.
\end{align*}
Thus the proof of Theorem \ref{t41} is completed.
\end{proof}

In this part of the section,  we consider the problem (4.1) with  
a more general elliptic operator $L$  and without the restriction 
$c(x)\le 0,\ x\in  \overline{\Omega}$. 
More precisely, in place of (4.1)  we consider the problem
$$
\left\{
\begin{array}{rl}
&\ppp_t^{\alpha}u(x,t) = -Lu(x,t) +  F(x,t), \quad x\in \Omega, 
\quad 0 < t < T,\\
&u(x,0) = 0, \quad x \in \Omega,\\
&u\vert_{\ppp\Omega\times (0,T)} = 0,
\end{array} \right.  \eqno{(4.10)}
$$
where we set 
$$
Lu(x,t) = -\sum_{j,k=1}^n \frac{\partial }{\partial x_j}\left( 
a_{jk}(x)\frac{\partial }{\partial x_k} u(x,t)\right) - \sum_{j=1}^n 
b_{j}(x)\frac{\partial }{\partial x_j} u(x,t) - c(x)u(x,t),
$$
and we assume the condition (4.2) to be fulfilled, $b_j \in L^{\infty}
(\Omega)$, $1\le j\le n$ and $c \in L^{\infty}(\Omega)$.  
Note that the condition
$c(x)\le 0,\ x \in \overline{\Omega}$ is not assumed any more.

The weak solution to the problem (4.10) is defined in the exactly same way as the one to the equation (4.1). Our main result regarding the problem (4.10) is as follows:

%\vspace{0.1cm}

%\noindent
%{\bf Theorem 4.2.}\\
\begin{thm}
\label{t42}
Let $F \in L^2(0,T;L^2(\Omega))$ and $0 < \alpha < 1$.
Then there exists a unique weak solution to the problem (4.10).
Moreover, there exists a constant $C>0$ such that the norm estimate 
(4.4) holds true for all $F \in L^2(0,T;L^2(\Omega))$.
\end{thm}
\begin{proof}
For a function $a \in L^2(\Omega)$, let us define the operator 
$$
K(t)a = \sum_{k=1}^{\infty} t^{\alpha-1} E_{\alpha,\alpha}
(-\mu_kt^{\alpha})(a,\va_k)_{L^2(\Omega)}\va_k, \quad t > 0.
$$
Following the lines of \cite{Beck}, one can easily verify the norm estimate
$$
\Vert K(t)a\Vert_{\LLL} \le Ct^{\alpha-1}\Vert a\Vert_{\LLL}, \quad
a\in \LLL, \quad t>0.                               \eqno{(4.11)}
$$
For $0 \le \gamma \le 1$ one can define the fractional power 
$L^{\gamma}$ of the operator $L$ defined by (4.5).  Then according to \cite{F},
the inequalities 
$$
\left\{
\begin{array}{rl}
&\Vert v\Vert_{H^2(\Omega)} \le C\Vert Lv\Vert_{\LLL},
\quad v \in H^2(\Omega) \cap H^1_0(\Omega), \\
&C^{-1}\Vert L^{\frac{1}{2}}v\Vert_{\LLL} \le 
\Vert v\Vert_{H^1(\Omega)} \le C\Vert L^{\frac{1}{2}}v\Vert_{\LLL},
\quad v \in H^1_0(\Omega)
\end{array}\right.
\eqno{(4.12)}
$$
hold true.  
Moreover, for $0 \le \gamma \le 1$, we have
$L^{\gamma}K(t)a = K(t)L^{\gamma}a$ for $a \in \mathcal{D}(L^{\gamma})$,
and 
$$
\Vert L^{\gamma}K(t)a\Vert_{\LLL} \le Ct^{\alpha(1-\gamma)-1}
\Vert a\Vert_{\LLL}, \quad 0 \le \gamma \le 1, \quad t>0.         \eqno{(4.13)}
$$
Now we interpret the function $- \sum_{j=1}^n 
b_{j}(x)\frac{\partial }{\partial x_j} u(x,t) - c(x)u(x,t)$ as a 
non-homogeneous term in the equation (4.1) and apply 
Theorem \ref{t41}, so that we have a weak solution $u(t) := u(\cdot,t)$ of 
the problem (4.10) in the form 
$$
u(t) = -\int^t_0 K(t-s)\left( \sum_{j=1}^n b_j\ppp_ju(s) + cu(s)\right) ds
+  \int^t_0 K(t-s)F(s) ds, \quad t > 0.             \eqno{(4.14)}
$$
First we prove the uniqueness of the weak solution.  Let $F=0$ in (4.14).
Then, since $u(\cdot,t) \in H^2(\Omega) \cap H^1_0(\Omega)$ for 
$t >0$, by (4.12) and (4.13) we obtain 
\begin{align*}
& \Vert L^{\frac{1}{2}}u(t)\Vert_{\LLL}
\le C\left\Vert \int^t_0 L^{\frac{1}{2}}K(t-s)
\left( \sum_{j=1}^n b_j\ppp_ju(s) + cu(s)\right) ds \right\Vert_{\LLL}\\
\le &C \int^t_0 (t-s)^{\frac{1}{2}\alpha-1} 
\left(\left\Vert \sum_{j=1}^n b_j\ppp_ju(s)\right\Vert_{\LLL} 
+ \Vert cu(s)\Vert_{\LLL} \right) ds.
\end{align*}
Therefore (4.12) yields
$$
\Vert u(t)\Vert_{H^1(\Omega)} \le C\int^t_0 (t-s)^{\frac{1}{2}\alpha-1}
\Vert u(s)\Vert_{H^1(\Omega)} ds, \quad t > 0.
$$
The generalized Gronwall inequality (Theorem 7.2.1 in 
\cite{H}) yields $u(t) = 0$, $0 < t < T$ that completes  the proof of 
the uniqueness of the weak solution.

Next the existence of the weak solution is proved.  First an operator $Q$
from $L^2(0,T;H^2(\Omega))$ to itself is introduced by 
$$
Qu(t) = -\int^t_0 K(t-s)\left( \sum_{j=1}^n b_j\ppp_ju(s) + cu(s)\right) ds,
\quad 0 < t < T.
$$
Taking into consideration  Theorem \ref{t41}, it is sufficient to prove that the 
equation $u = Qu + G(t)$ has a unique solution in $L^2(0,T;H^1_0(\Omega))$.
Here we set $G(t) = \int^t_0 K(t-s)F(s) ds$ and Theorem \ref{t41} yields
$$
\Vert G\Vert_{L^2(0,T;H^2(\Omega))} 
+ \Vert G\Vert_{H^{\alpha}(0,T;L^2(\Omega))} 
\le C\Vert F\Vert_{L^2(0,T;\LLL)}.           \eqno{(4.15)}
$$
The estimates (4.12) and (4.13) lead to the inequality 
$$
\Vert L^{\frac{1}{2}}Qu(t)\Vert_{\LLL}
= \left\Vert -\int^t_0 L^{\frac{1}{2}}K(t-s)
\left( \sum_{j=1}^n b_j\ppp_ju(s) + cu(s)\right) ds\right\Vert_{\LLL}
$$
$$
\le C\int^t_0 (t-s)^{\frac{1}{2}\alpha-1}\Vert L^{\frac{1}{2}}u(s)
\Vert_{\LLL} ds,
\quad 0 < t < T.                       \eqno{(4.16)}
$$
Applying (4.16), we obtain the following chain of the inequalities:  
\begin{align*}
& \Vert L^{\frac{1}{2}}Q^2u(t)\Vert_{\LLL} 
= \Vert L^{\frac{1}{2}}Q(Qu(t))\Vert_{\LLL}\\
\le& C\int^t_0 (t-s)^{\frac{1}{2}\alpha-1}\Vert L^{\frac{1}{2}}(Qu(s))
\Vert_{\LLL} ds
\le C^2\int^t_0 (t-s)^{\frac{1}{2}\alpha-1}
\left( \int^s_0 (s-\xi)^{\frac{1}{2}\alpha-1}
\Vert L^{\frac{1}{2}}u(\xi)\Vert_{\LLL}d\xi
\right) ds\\
=& C^2\int^t_0 \left( \int^t_{\xi} (t-s)^{\frac{1}{2}\alpha-1}
(s-\xi)^{\frac{1}{2}\alpha-1} ds \right) 
\Vert L^{\frac{1}{2}}u(\xi)\Vert_{\LLL}d\xi\\
=& \frac{\left( C\Gamma\left(\frac{1}{2}\alpha\right)\right)^2}
{\Gamma(\alpha)}
\int^t_0 (t-\xi)^{\alpha-1} \Vert L^{\frac{1}{2}}u(\xi)\Vert_{\LLL}d\xi.
\end{align*}
Repeating the last estimation $m$-times, we obtain the inequality
$$
\Vert L^{\frac{1}{2}}Q^mu(t)\Vert_{\LLL}
\le \frac{\left( C\Gamma\left(\frac{1}{2}\alpha\right)\right)^m}
{\Gamma\left(\frac{1}{2}m\alpha\right)}
\int^t_0 (t-s)^{\frac{m}{2}\alpha-1} \Vert L^{\frac{1}{2}}u(s)\Vert_{\LLL}ds,
\quad 0<t<T, \thinspace m \in \N.
$$
Now we choose $m\in \N$ such that $\frac{m}{2}\alpha - 1> 0$ and set 
$C_m = \frac{\left( C\Gamma\left(\frac{1}{2}\alpha\right)\right)^m}
{\Gamma\left(\frac{1}{2}m\alpha\right)}$.
Then
$$
\Vert Q^mu(t)\Vert_{H^1(\Omega)} 
\le C_m\int^t_0 \max_{0\le t\le t} (t-s)^{\frac{m}{2}\alpha-1}
\Vert u(s)\Vert_{H^1(\Omega)} ds
\le T^{\frac{m}{2}\alpha-1}C_m\int^t_0 \Vert u(s)\Vert_{H^1(\Omega)} ds.
$$
Hence, setting $\rho_m = T^{\frac{m}{2}\alpha-1}C_m$, we arrive at the estimate
$$
\Vert Q^mu(t)\Vert^2_{H^1(\Omega)} 
\le \rho_m^2\left(\int^t_0 \Vert u(s)\Vert_{H^1(\Omega)} ds\right)^2
\le \rho_m^2T^2\int^T_0 \Vert u(s)\Vert^2_{H^1(\Omega)} ds,
$$
which implies the inequality
$$
\int^T_0 \Vert Q^mu(t)\Vert^2_{H^1(\Omega)} dt 
\le \rho_m^2T^2\int^T_0 \Vert u(s)\Vert^2_{H^1(\Omega)} ds.
$$
By the asymptotic behavior of the gamma function, it is easy to verify that 
$$
\lim_{m\to\infty} \rho_m 
= T^{-1}\lim_{m\to\infty}  \frac{\left( T^{\frac{\alpha}{2}}
C\Gamma\left(\frac{1}{2}\alpha\right)\right)^m}
{\Gamma\left(\frac{1}{2}m\alpha\right)} = 0.           \eqno{(4.17)}
$$
Hence $T\rho_m < 1$ for large $m\in \N$.
Now we set $\widetilde{Q}u = Qu + G$ and $w=u-v$. Then
$Qw = \widetilde{Q}w$ and $Q^mw = \widetilde{Q}^mw$, and it follows
from (4.17) that 
$\widetilde{Q}^m$ is a contraction from $L^2(0,T;H^1(\Omega))$ to 
itself.  Hence the mapping $\widetilde{Q}^m$ has a unique fixed point 
$u_*\in L^2(0,T;H^1(\Omega))$, that is, $\widetilde{Q}^mu_* = u_*$.  
Because $\widetilde{Q}^m(\widetilde{Q}u_*) = \widetilde{Q}u_*$, the point 
$\widetilde{Q}u_*$ is also a fixed point of the mapping $\widetilde{Q}^m$.  
By the uniqueness of the fixed point 
of $\widetilde{Q}^m$, we finally see the equality 
$u_* = \widetilde{Q}u_* = Qu_* + G$.
Thus the equation $u=Qu+G$ has a unique solution in $L^2(0,T;H^1_0(\Omega))$ and
$\Vert u\Vert_{L^2(0,T;H^1(\Omega))} \le C\Vert G\Vert
_{L^2(0,T;H^1(\Omega))}$.

Moreover, (4.15) implies
$$
\Vert u\Vert_{L^2(0,T;H^1(\Omega))} \le C\Vert G\Vert_{L^2(0,T;
H^1(\Omega))} \le C\Vert F\Vert_{L^2(0,T;\LLL)}. 
$$
Therefore $\left\Vert \sum_{j=1}^n b_j\ppp_ju + cu\right\Vert
_{L^2(0,T;\LLL)} \le C\Vert F\Vert_{L^2(0,T;\LLL)}$ and so Theorem \ref{t41}
yields the estimate
\begin{align*}
& \left\Vert Q\left(\sum_{j=1}^n b_j\ppp_ju + cu\right)\right\Vert
_{L^2(0,T;H^2(\Omega)) \cap H^{\alpha}(0,T;\LLL)}\\
= &\left\Vert \int^t_0 K(t-s)\left(\sum_{j=1}^n b_j\ppp_ju(s) + cu(s)
\right)ds\right\Vert
_{L^2(0,T;H^2(\Omega)) \cap H^{\alpha}(0,T;\LLL)}\\
\le& C\left\Vert \sum_{j=1}^n b_j\ppp_ju + cu\right\Vert
_{L^2(0,T;\LLL)}
\le C\Vert F\Vert_{L^2(0,T;\LLL)},
\end{align*}
which proves (4.4) and so the proof of the theorem is completed.
\end{proof}

\end{document}